\documentclass{amsart}
\usepackage[margin=1in]{geometry}
\usepackage{amsmath}
\usepackage{amssymb}
\usepackage{amsthm}
\usepackage{amscd}
\usepackage{enumerate}

\usepackage{graphicx}
\usepackage{amsmath,amsthm,amsfonts,amscd,amssymb,comment,eucal,latexsym,mathrsfs}
\usepackage{stmaryrd}
\usepackage[all]{xy}

\usepackage{epsfig}

\usepackage[all]{xy}
\xyoption{poly}
\usepackage{fancyhdr}
\usepackage{wrapfig}
\usepackage{epsfig}

\theoremstyle{plain}
\newtheorem{thm}{Theorem}[section]
\newtheorem{prop}[thm]{Proposition}
\newtheorem{lem}[thm]{Lemma}

\theoremstyle{definition}

\theoremstyle{remark}

\advance \topmargin by -\headheight
\advance \topmargin by -\headsep
\evensidemargin \oddsidemargin
  \def\C{{\mathbb{C}}}           \def\N{{\mathbb{N}}}    \def\R{{\mathbb{R}}}  \def\T{{\mathbb{T}}}      \def\Z{{\mathbb{Z}}}

\newcommand\ev{\operatorname{ev}}

\newcommand\id{\operatorname{id}}

\renewcommand\Im{\operatorname{Im}}

\newcommand\Proj{\operatorname{Proj}}
\newcommand\Prob{\operatorname{Prob}}

\def\cc{{\curvearrowright}}
\newcommand{\actson}{\curvearrowright}
\newcommand{\actons}{\curvearrowright}

\newcommand{\ip}[1]{\langle #1 \rangle}

\begin{document}
\title{Weak Equivalence to Bernoulli Shifts for Some Algebraic Actions}      
\author{Ben Hayes}
\address{University of Virginia\\
         Kerchof Hall\\
         Charlottesville, VA 22904}
\email{brh5c@virginia.edu}\thanks{The author gratefully acknowledges support by  NSF Grant DMS-1600802.}
\date{\today}
\maketitle

\begin{abstract} Let $G$ be a countable, discrete group and $f\in M_{n}(\Z(G)).$ We continue our study of the connections between operator theoretic properties of $f$ as a convolution operator on $\ell^{2}(G)^{\oplus n}$ and the ergodic theoretic properties of the action of $G$ on the Pontryagin dual of $(\Z(G)^{\oplus n}/\Z(G)^{\oplus n}f),$ (denoted $X_{f}$). Namely, we prove that if $G$ is a countable, discrete group and $f\in M_{n}(\Z(G))$ is invertible on $\ell^{2}(G)^{\oplus n},$ but $f$ is not invertible in $M_{n}(\Z(G))$, then the measure-preserving action of $G$ on $X_{f}$ equipped with the Haar measure is weakly equivalent to a Bernoulli action. This explains some of the ``Bernoulli-like'' properties that  $G\cc X_{f}$ has. We shall in fact prove this weak equivalence in the case that $f$ has a ``formal inverse in $\ell^{2}$''.
\end{abstract}

\tableofcontents
\section{Introduction}
Let $G$ be a countable, discrete group. An \emph{algebraic action} of $G$ is an action $G\cc X$ by continuous automorphisms, where $X$ is a compact, metrizable, abelian group. We often consider this action as a probability measure-preserving action, giving $X$ the Haar measure $m_{X}$. The goal of this paper is to give examples of algebraic actions which are, in a precise sense, similar to Bernoulli shifts. In particular, we shall give examples related to invertible convolution operators. Given $f\in M_{m,n}(\C(G)),$  write $f_{lk}=\sum_{h\in G}\widehat{f_{lk}}(h)h$ for $1\leq l\leq m$, $1\leq k\leq n.$ We  define operators $\lambda(f)\colon c_{0}(G)^{\oplus n}\to c_{0}(G)^{\oplus m},$ and $r(f)\colon c_{o}(G)^{\oplus m}\to c_{0}(G)^{\oplus n}$ by
\[(\lambda(f)\xi)(l)(g)=\sum_{k=1}^{n}\sum_{h\in G}\widehat{f_{lk}}(h)\xi(k)(h^{-1}g), \mbox{for $g\in G,1\leq l \leq m,\xi\in c_{0}(G)^{\oplus n}$} \]
\[(r(f)\xi)(l)(g)=\sum_{k=1}^{m}\sum_{h\in G}\xi(k)(h)\widehat{f_{kl}}(h^{-1}g),\mbox{for $g\in G,1\leq l\leq n,\xi\in c_{0}(G)^{\oplus m}.$}\]

Note that $\lambda(f)(\ell^{p}(G)^{\oplus n})\subseteq \ell^{p}(G)^{\oplus m}$ for every $1\leq p\leq \infty,$  so we may also regard $\lambda(f)$ as an operator $\ell^{p}(G)^{\oplus n}\to \ell^{p}(G)^{\oplus m}.$ Similar remarks apply to $r(f).$ We also define $r(f)\colon \C(G)^{\oplus m}\to \C(G)^{\oplus n}$ by
\[(r(f)\alpha)(l)=\sum_{k=1}^{m}\alpha_{k}f_{kl}.\]
We let $X_{f}$ be the Pontryagin dual of $\Z(G)^{\oplus n}/r(f)(\Z(G)^{\oplus m})$, i.e. $X_{f}$ is the space of  group homomorphisms from $\Z(G)^{\oplus n}/r(f)(\Z(G)^{\oplus m})$ to $\T=\R/\Z$. We then have a natural action of $G$ on $X_{f}$ by
\[(g\theta)(a)=\theta(g^{-1}a),\mbox{ for $a\in \Z(G)^{\oplus n}/r(f)(\Z(G)^{\oplus m}),\theta\in X_{f},g\in G$}.\]
When $m=n=1,$ this is called a \emph{principal algebraic action} and has been studied by many authors (see \cite{BowenEntropy},\cite{Den},\cite{DenSchmidt},\cite{KerrLi2},\cite{LiSchmidtPet},\cite{LiThom},\cite{LindSchmidt1},\cite{LindSchmidt2}). We will call the case of $m=n$ (but not necessarily $m=n=1$) a \emph{balanced algebraic action}, essentially all of what can be said about principal algebraic actions can be said for balanced algebraic actions (see e.g. \cite{Me13},\cite{Me5},\cite{Me7},\cite{KerrLi2}). A principal theme through much of the recent work studying algebraic actions is that ergodic theoretic properties of $G\actson (X_{f},m_{X_{f}})$ (e.g. ergodicity, complete positive entropy when $G$, expansiveness) translate into operator theoretic properties of $\lambda(f).$ For example, when $m=n$ and  $\lambda(f)$ is invertible as an operator $\ell^{2}(G)^{\oplus n}\to \ell^{2}(G)^{\oplus n}$ the action $G\actson (X_{f},m_{X_{f}})$ shares many properties with Bernoulli shifts: it is ergodic (\cite{LiSchmidtPet}), in fact mixing (\cite{ChungLi}), the Koopman representation is isomorphic to an infinite direct sum of the left regular representation, and it has completely positive entropy when $G$ is assumed to be sofic (\cite{Me13}). We remark that if $m\ne n$ and $f\in M_{m,n}(\Z(G)),$ then it is not possible that $\lambda(f)$ is invertible $\ell^{2}(G)^{\oplus n}\to \ell^{2}(G)^{\oplus m}$  (this follows e.g. from \cite{Luck} Lemma 1.13 (5)). 

	The goal of this paper is to give a precise statement of the similarity of these actions to Bernoulli shifts: namely we show that they are \emph{weakly equivalent} to Bernoulli shifts.   Weak equivalence is related to the notion of weak containment introduced by Kechris (see Chapter II.10 of \cite{KechrisGA}). Suppose we are given a probability measure-preserving action $G\cc (X,\mu)$ and a measurable map $\phi\colon X\to A$ where $A$ is finite. For a finite $F\subseteq G$, define $\phi^{F}\colon X\to A^{F}$ by $\phi^{F}(x)(g)=\phi(g^{-1}x).$ We say that $G\cc (X,\mu)$ is weakly contained in another action $G\cc (Y,\nu)$ if for every finite set $A,$  every measurable $\phi\colon X\to A$, every finite $F\subseteq G,$ and every $\varepsilon>0,$ there is a measurable $\psi\colon Y\to A$  so that
\[\|\phi^{F}_{*}(\mu)-\psi^{F}_{*}(\nu)\|<\varepsilon,\]
where the norm in question is the total variation norm. Intuitively, any ``finitary piece'' of $G\cc (X,\mu)$ can be approximated arbitrarily well by some ``finitary piece'' of $G\cc (Y,\nu)$ (in many ways, this definition is analogous to weak containment of representations or finite representability of Banach spaces). We then say $G\cc (X,\mu),G\cc (Y,\nu)$ are weakly equivalent if $G\cc (X,\mu)$ is weakly contained in $G\cc (Y,\nu)$ and $G\cc (Y,\nu)$ is weakly contained in $G\cc (X,\mu).$ We mention that much of the interest of this problem comes from the question of whether or not $G\cc (X_{f},m_{X_{f}})$ is isomorphic to a Bernoulli shift when $f$ is invertible in $\ell^{1}(G)$ (see Conjecture 6.8 in \cite{LindSchmidt2} where this is conjectured when $G=\Z^{d}$). As mentioned before, many works have already shown that $G\cc (X_{f},m_{X_{f}})$ shares many properties of Bernoulli shift actions when $\lambda(f)$ is invertible  as an operator $\ell^{2}(G)^{\oplus n}\to \ell^{2}(G)^{\oplus n}$ (which is weaker than $f$ being invertible in $\ell^{1}(G)$).

We will actually prove something more general than just saying that $G\cc (X_{f},m_{X_{f}})$ is weakly contained in Bernoulli shifts if $\lambda(f)$ is invertible, we will prove this weak containment when $\lambda(f)$ just has an ``$\ell^{2}$ formal inverse.''
 Given $\xi\in M_{n}(\ell^{2}(G)),\alpha\in M_{n}(\C(G))$ we define $\alpha\xi,\xi\alpha\in M_{n}(\ell^{2}(G))$ by
\[(\alpha\xi)_{ij}=\sum_{k=1}^{n}\lambda(\alpha)_{ik}\xi_{kj}, (\xi \alpha)_{ij}=\sum_{k=1}^{n}r(\alpha)_{kj}\xi_{ik}. \]
We say that $x\in M_{n}(\C(G))$ has an \emph{$\ell^{2}$ formal inverse} if there is a $\xi\in M_{n}(\ell^{2}(G)))$ so that
\begin{equation}\label{E:inverseintro}
(x\xi)_{ij}=\delta_{i=j}\delta_{e},\,\,\,\,\, (\xi x)_{ij}=\delta_{i=j}\delta_{e},
\end{equation}
The concept of an $\ell^{2}$ formal inverse is related to the concept of a $c_{0}$ formal inverse that appears in \cite[Example 4.7]{BowenLi}. To the best of our knowledge, this is the first time this has appeared in the literature. As discussed in Section \ref{S:L2examples}, it is related to the algebra of measurable operators affiliated to the group von Neumann algebra defined, for example, in \cite{TakesakiII} IX.2 . This algebra is often used in functional analysis, but we will not need it for this paper. In Section \ref{S:L2examples}, we will show that if $\xi\in M_{n}(\ell^{2}(G))$ and either one of the equations in (\ref{E:inverseintro}) hold, then $\xi$ is an $\ell^{2}$ formal inverse to $x.$ 
From this, it follows that $x$ has an $\ell^{2}$ formal inverse if and only if $x(\ell^{2}(G)^{\oplus n})\supseteq c_{c}(G)^{\oplus n}.$ It is also true that if an $\ell^{2}$ formal inverse exists, then it is unique (see Section \ref{S:L2examples}).  Our main theorem is the following.

\begin{thm}\label{T:mainintro} Let $G$ be a countably infinite, discrete group and let $f\in M_{n}(\Z(G)).$ Suppose that $f$ has an $\ell^{2}$ formal inverse, but that $f\notin GL_{n}(\Z(G)).$ Then $G\cc (X_{f},m_{X_{f}})$ is weakly equivalent to a Bernoulli shift.
\end{thm}

Note that this covers the case when $\lambda(f)$ is invertible in $B(\ell^{2}(G)^{\oplus n}).$
One direction of the proof of Theorem \ref{T:mainintro} relies on a result of Ab\'{e}rt-Weiss in \cite{AbertWeiss}, as well as the fact that the action is free modulo its kernel (which is a finite group). This last fact follows from the fact that these actions are mixing (essentially following from \cite{ChungLi} Lemma 5.4 or \cite{BowenLi} Theorem 4.1) as well as a stunning theorem of Tucker-Drob (see \cite{RobinAF}), which ultimately relies on the Feit-Thompson thoerem. In an earlier draft of this paper, we wrote an elementary proof of the fact that if $f\in M_{n}(\Z(G))$ has an $\ell^{2}$ formal inverse, then  $G\actson (X_{f},m_{X_{f}})$ is free modulo its kernel. We have elected not to include this proof so as to not introduce too many technalities that may distract from the main result of the paper.

To illustrate the breadth of our results, we mention a few examples of $f\in \Z(G)$ which have $\lambda(f)$ invertible. This is the case if $f$ is invertible in $\ell^{1}(G),$ which, by standard Banach algebra techniques, applies if $f=b-\sum_{x}a_{x}x$ with $\sum_{x}|a_{x}|<|b|.$ For examples where $\lambda(f)$ is invertible, but $f$ is not invertible in $\ell^{1}(G),$ suppose that $G$ is nonamenable and again consider $f=b-\sum_{x}a_{x}x,$  but now suppose that $\{y^{-1}x:x,y\in G,a_{x}\ne 0,a_{y}\ne 0\}$ generates a nonamenable subgroup of $G,$ and that $\sum_{x}|a_{x}|=b.$ In this case, it is well known that
\[\left\|\sum_{x}a_{x}\lambda(x)\right\|<b\]
and this again implies that $\lambda(f)$ is invertible. In  general it is not clear if $f$ is invertible in $\ell^{1}$ or not. However, if we  additionally assume that $a_{x}\geq 0$ for all $x\in G,$ then consideration of the homomorphism $\ell^{1}(G)\to \C$ defined by $\alpha\mapsto \sum_{g}\alpha_{g}$ shows that $f$ is not invertible in $\ell^{1}(G)$. This example is called the Harmonic model, since $X_{f}$ may be regarded as the space of Harmonic functions $f\colon G\to \R/\Z.$ The Harmonic model was studied in \cite{BowenLi} and it is related to Wired Spanning Forests and tree entropy as defined by Lyons in \cite{Lyons}.
Another example can be given as follows: assume there are $g,h\in G$ so that $g,h$ such that the semigroup generated by $g,h$ (but not necessarily the group generated by $g,h$) is a free semigroup. It then follows from Example A.1 of \cite{Li2} that
\[f=3e+(e-g-g^{2})h\]
has the property that $\lambda(f)$ is invertible, but $f$ is not invertible in $\ell^{1}(G).$

\textbf{Acknowledgments.} I thank the anonymous referee for their comments and suggestions, which improved the paper. I thank Lewis Bowen for interesting discusssions related to this problem. I thank Brandon Seward for suggesting that it may be interesting to find new examples of algebraic actions weakly equivalent to Bernoulli shifts.

\section{Preliminaries on $\ell^{2}$ formal inverses}\label{S:L2examples}

We start by introducing some notation. For any vector space $V,$ we wil think of $V^{\oplus n}$ as all functions $\{1,\dots,n\}\to V.$   For $f\in M_{n}(\C(G)),$ we let $\lambda(f),r(f)\in B(\ell^{2}(G)^{\oplus n})$ be defined as in the introduction. For $v\in V$ and an integer $n,$ we let $v\otimes\id_{n}\in M_{n}(V)$ be defined by $(v\otimes \id_{n})_{ij}=\delta_{i=j}v$ for $1\leq i,j\leq n.$ If the integer $n$ is clear from context, we will often write $\id$ instead of $\id_{n}.$  

We let $L(G)$ be the closure of $\lambda(\C(G))$ in the strong operator topology. Define $J\colon \ell^{2}(G)\to \ell^{2}(G)$ by $(J\xi)(g)=\overline{\xi(g^{-1})}.$  Given $x\in L(G),\xi\in \ell^{2}(G),$ we set $\xi x=J(xJ\xi).$ Note that if $f\in \C(G),$ then $\xi \lambda(f)=r(f)(\xi),$ and so this agrees with right multiplication by $f.$
For $x\in M_{n}(\ell^{2}(G)),$ we let $\widehat{x}=x\cdot (\delta_{e}\otimes \id).$ Note that $(\widehat{x})_{ij}=\widehat{x_{ij}}.$ We extend this notation to $M_{n}(\C(G))$ by setting $\widehat{\alpha}=\widehat{\lambda(\alpha)}$ for $\alpha\in M_{n}(\C(G)).$ For $\alpha\in \C(G)^{\oplus n},$ we define $\widehat{\alpha}\in \ell^{2}(G)^{\oplus n}$ by $\widehat{\alpha}(j)=\widehat{\alpha(j)}.$ Given $\xi\in M_{n}(\ell^{2}(G)),\alpha\in M_{n}(L(G))$ we define $\xi\alpha\in M_{n}(\ell^{2}(G)),\alpha\xi\in M_{n}(\ell^{2}(G))$ by
 \[(\xi\alpha)_{ij}=\sum_{k=1}^{n}\xi_{ik}\alpha_{kj},\,\, (\alpha\xi)_{ij}=\sum_{k=1}^{n}\alpha_{ik}\xi_{kj}\mbox{ for $1\leq i,j\leq n.$}\]
 We define $\|\cdot\|_{2}$ on $M_{n}(\ell^{2}(G))$ by $\|\xi\|_{2}^{2}=\sum_{1\leq i,j\leq n}\|\xi_{ij}\|_{2}^{2}.$ For $\xi\in M_{n}(\ell^{2}(G)),$ we define $\xi^{*}\in M_{n}(\ell^{2}(G))$ by $(\xi^{*})_{ij}(g)=\overline{\xi_{ji}(g^{-1})}.$

For $\xi\in M_{n}(\ell^{2}(G))$ we define $r(\xi)\colon \C(G)^{\oplus n}\to \ell^{2}(G)^{\oplus n}$ by
\[(r(\xi)\alpha)(l)=\sum_{k=1}^{n}\lambda(\alpha(k))\xi_{kl}.\]
It is straightforward to see that $r(f)r(\xi)=r(\xi\lambda(f))$ and $r(\xi)r(f)=r(\lambda(f)\xi)$ for $\xi\in M_{n}(c_{0}(G)),f\in M_{n}(\C(G)).$
It is easy to see that if $\alpha\in \C(G),f\in M_{m,n}(\C(G)),$ then
\[r(f)(\lambda(\alpha) \xi)=\lambda(\alpha) (\xi\lambda(f)),\lambda(f)(\xi\alpha)=(\lambda(f)\xi)\alpha\]
and that if $f_{1}\in M_{m,n}(Z(G)),f_{2}\in M_{n,k}(\Z(G)),$ then
\[\lambda(f_{1}f_{2})=\lambda(f_{1})\lambda(f_{2}), r(f_{1}f_{2})=r(f_{2})r(f_{1}).\]
We let $X_{f}$ be the Pontryagin dual of $\Z(G)^{\oplus n}/r(f)(\Z(G)^{\oplus m}).$  For $\alpha\in \Z(G)^{\oplus n}$ and $\theta\in (\T^{G})^{\oplus n}$ we define
\[\ip{\theta,\alpha}_{\T}=\sum_{1\leq l\leq n}\sum_{g\in G}\widehat{\alpha}(l)(g)\theta(l)(g).\]
This duality identifies $(\T^{G})^{\oplus n}$ with the Pontryagin dual of $\Z(G)^{\oplus n}.$

We now  prove some basic properties of $\ell^{2}$ formal inverses. One of our crucial results is that if $\lambda(f)\in M_{n}(\C(G)),\xi\in M_{n}(\ell^{2}(G))$ and $\lambda(f)\xi=\delta_{e}\otimes \id,$ then automatically $\xi \lambda(f)=\delta_{e}\otimes \id.$  We remark that essentially all of this can be proved by working in the algebra of measurable operators affiliated to $L(G)$ (see \cite{TakesakiII} IX.2 for the necessary definition). Indeed, having an $\ell^{2}$ formal inverse is equivalent to saying that $\lambda(f)$ is injective and that its inverse (as an unbounded operator) is in the noncommutative $L^{2}$-space of $L(G)$ with respect to the trace. To avoid this (somewhat technical) background, we will  reprove some known facts about $\ell^{2}$ inverses in a more elementary way (in particular, both Propositions \ref{P:basicpropertiesofformalinverse} and \ref{P:basics} can  be derived from known properties of the algebra of measurable operators).

 \begin{prop}\label{P:basics} Let $G$ be a countable, discrete group.
  \begin{enumerate}[(i)]
  \item  \label{I:fourierswitch} For any $\alpha,\beta\in M_{n}(L(G))$ we have that $\widehat{\alpha}\beta=\widehat{\alpha\beta}=\alpha\widehat{\beta}.$
  \item \label{I:fourier continuity} If  we have a sequence $x_{n}\in M_{n}(L(G))$ and $x_{n}\to x$ in the strong operator topology, then $\|\widehat{x_{n}}-\widehat{x}\|_{2}\to 0.$ 
  \item If $x\in M_{n}(L(G))$ and $\widehat{x}=0,$ then $x=0.$ \label{I:fourier injectivity}
   \item For $\xi\in M_{n}(\ell^{2}(G)),\alpha,\beta\in \C(G)^{\oplus n}$ we have $\ip{r(\xi)\alpha,\widehat{\beta}}=\ip{\widehat{\alpha},r(\xi^{*})\beta}.$ \label{I:l2adjoints}
  \end{enumerate}

 \end{prop}

 \begin{proof}
 (\ref{I:fourierswitch}): The fact that $\alpha\widehat{\beta}=\widehat{\alpha\beta}$ is tautological, so we only present the proof that $\widehat{\alpha}\beta=\widehat{\alpha\beta}.$  First observe  that $\delta_{e}\lambda(f)=\widehat{\lambda(f)}$ for any $f\in \C(G).$ It follows that $\delta_{e}\widehat{x}=\widehat{x}$ for any $x\in L(G)$ by approximating $x$ in the strong operator topology by elements of $\C(G).$ This implies (\ref{I:fourierswitch}) when $\alpha,\beta\in L(G).$ The general case reduces to the case $\alpha,\beta\in L(G)$ by a direct calculation.

 (\ref{I:fourier continuity}): This is automatic from the fact that if $T_{n}$ is a sequence in $M_{n}(B(\ell^{2}(G))$ and $T$ is any element of $M_{n}(B(\ell^{2}(G))),$ then $T_{n}\to T$ in the strong operator topology if and only if $(T_{n})_{ij}\to T_{ij}$ in the strong operator topology for all $1\leq i,j\leq n.$

 (\ref{I:fourier injectivity}): Since $\widehat{x}_{ij}=\widehat{x_{ij}},$ it is enough to handle the case $n=1.$ In this case, $\|\widehat{x}\|_{2}^{2}=\tau(x^{*}x),$ where $\tau\colon L(G)\to \C$ is defined by $\tau(x)=\ip{x\delta_{e},\delta_{e}}.$ By (\cite{Luck}, Theorem 1.9 (3)), we see that $\|\widehat{x}\|_{2}^{2}=0$ if and only if $x=0.$

(\ref{I:l2adjoints}): By a direct calculation, the case of general $n$ reduces to the case $n=1,$ so we assume that $n=1.$ It is straightforward to see that  for any $c\in \C(G)$  we have that $\ip{\xi,\widehat{c}}=\ip{\widehat{c^{*}},\xi^{*}}.$ Thus,
\[\ip{r(\xi)\alpha,\widehat{\beta}}=\ip{\lambda(\alpha)\xi,\lambda(\beta)\delta_{e}}=\ip{\xi,\lambda(\alpha^{*}\beta)\delta_{e}}=\ip{\lambda(\beta^{*}\alpha)\delta_{e},\xi^{*}}=\ip{\lambda(\alpha)\delta_{e},\lambda(\beta)\xi^{*}}=\ip{\widehat{\alpha},r(\xi^{*})\beta}.\]

 \end{proof}

 We present some basic properties of $\ell^{2}$ formal inverses. If $T\in B(\ell^{2}(G)^{\oplus n})$ is a normal operator and $\phi$ is a bounded Borel function, we let $\phi(T)$ be the operator defined  by Borel functional calculus (see \cite{Conway} IX.2 for the necessary background). If $T\in B(\ell^{2}(G)^{\oplus n})$ is an arbitrary operator, we let $|T|=(T^{*}T)^{1/2}.$ Note that these comments apply to elements of $M_{n}(L(G)),$ since we may view $B(\ell^{2}(G)^{\oplus n})\cong M_{n}(B(\ell^{2}(G)))$ in a natural way.

\begin{prop}\label{P:basicpropertiesofformalinverse} Let $G$ be a countable, discrete, group and $f\in M_{n}(\C(G)).$ Suppose that $\xi\in M_{n}(\ell^{2}(G))$ and $\lambda(f)\xi=\delta_{e}\otimes \id.$ Then
\begin{enumerate}[(i)]
\item $\lambda(f)$ is injective as an operator $\ell^{2}(G)^{\oplus n}\to \ell^{2}(G)^{\oplus n}$,  \label{I:injectivity}
\item The operator $L_{\lambda(f)}\colon M_{n}(\ell^{2}(G))\to M_{n}(\ell^{2}(G))$ given by $L_{\lambda(f)}(\zeta)=\lambda(f)\zeta$ is injective, \label{I:injective2}
\item $\xi \lambda(f)=\delta\otimes \id,$ \label{I:left/right nonsense}
\item $\xi\in M_{n}(\ell^{2}(G,\R)).$  \label{I:tha realness}

\end{enumerate}

\end{prop}

\begin{proof}

(\ref{I:injectivity}): Suppose that $\lambda(f)$ is not injective. By the rank-nullity theorem for von Neumann dimension (\cite{Luck} Theorem 1.12 (2)), and the fact that $f$ is a square matrix, we know that $\dim_{L(G)}(\Im\lambda(f))^{\perp})=\dim_{L(G)}(\ker \lambda(f))>0.$ So if we let $p\in M_{n}(L(G))$ be the projection onto the orthogonal complement of the image of $\lambda(f),$ we then have that $p\ne 0.$ Since $p\lambda(f)=0,$
\[\widehat{p}=p(\delta_{e}\otimes \id)=p\lambda(f)\xi=(p\lambda(f))\xi=0.\]
Thus $\widehat{p}=0,$ so $p=0$  by Proposition \ref{P:basics} (\ref{I:fourier injectivity}) and we have a contradiction.

(\ref{I:injective2}): Define $\Phi\colon (\ell^{2}(G)^{\oplus n})^{\oplus n}\to M_{n}(\ell^{2}(G))$ by 
\[\Phi(\xi)_{ij}=\xi(j)(i)\mbox{ for $1\leq i,j\leq n$}\]
It is straightforward to check that $\Phi$ is a unitary and that 
\[\Phi\lambda(f)^{\oplus n}=\lambda(f)^{\oplus n}\Phi.\]
So this follows from (\ref{I:injectivity}).

(\ref{I:left/right nonsense}): Let $\lambda(f)=u|\lambda(f)|$ be the polar decomposition (see \cite{Conway} Proposition VIII.3.11), as is well known (see \cite{Taka} Proposition 3.14) we have that $u,|\lambda(f)|\in M_{n}(L(G)).$ Note $u\big|_{\ker(\lambda(f)}$ induces an isometry $(\ker\lambda(f))^{\perp}\to \overline{\Im\lambda(f)}$ and hence, by part (\ref{I:injectivity}), we know that $u$ is a unitary operator on $\ell^{2}(G)^{\oplus n}$. Let $p_{k}$ be a sequence of real polynomials with $|p_{k}(t)|\leq \frac{1}{t}$ for $t\in [-\|\lambda(f)\|,\|\lambda(f)\|]$ and so that $p_{k}(t)\to \frac{1}{t}$ uniformly on compact subsets of $[-\|\lambda(f)\|,\|\lambda(f)\|]\setminus\{0\}.$ Let $\phi_{k}=p_{k}(|\lambda(f)|)u^{*}.$ Observe that $\phi_{k}\lambda(f)=p_{k}(|\lambda(f)|)|\lambda(f)|,\lambda(f)\phi_{k}=up_{k}(|\lambda(f)|)|\lambda(f)|u^{*}.$ 
Since $\lambda(f)$ is injective, we know that $\chi_{\{0\}}(|\lambda(f)|)=\Proj_{\ker(\lambda(f))}=0,$ so by the spectral theorem we have that $\chi_{(0,\infty)}(|\lambda(f)|)=1.$ Another application of the spectral theorem shows that (with limits being taken in the strong operator topology):
\[\lim_{k\to\infty}\lambda(f)\phi_{k}=\lim_{k\to\infty}p_{k}(|\lambda(f)|)|\lambda(f)|=\chi_{(0,\infty)}(|\lambda(f)|)=1,\lim_{k\to\infty}\lambda(f)\phi_{k}=u\chi_{(0,\infty)}(|\lambda(f)|)u^{*}=1,\]
with the last equality following because $u$ is a unitary. 

Now set $\xi_{k}=\phi_{k}(\delta_{e}\otimes \id).$ By the above limiting formulas, we have that:
\[\|\xi_{k}-\xi\|=\|\phi_{k}(\delta_{e}\otimes \id)-\xi\|=\|(\phi_{k}\lambda(f)-\id)\xi\|\to_{k\to\infty}0.\]
So by Proposition \ref{P:basics} (\ref{I:fourierswitch}),
\[\xi\lambda(f)=\lim_{k\to\infty}\xi_{k}\lambda(f)=\lim_{k\to\infty}\widehat{\phi_{k}}\lambda(f)=\lim_{k\to\infty}\widehat{\phi_{k}\lambda(f)}=\delta_{e}\otimes \id,\]
the last equality following because $\phi_{k}\lambda(f)\to_{k\to\infty}\id$ in the strong operator topology.


(\ref{I:tha realness}): Define a conjugate-linear map $C\colon M_{n}(\ell^{2}(G))\to M_{n}(\ell^{2}(G))$ by $(C\zeta)_{ij}=\overline{\zeta_{ij}}$ for $1\leq i,j\leq n.$ Since $f\in M_{n}(\Z(G))\subseteq M_{n}(\R(G)),$ we have that  $L_{\lambda(f)}C=CL_{\lambda(f)}$ So, 
\[L_{\lambda(f)}(C\xi)=CL_{\lambda(f)}\xi=C\lambda(f)\xi=C(\delta_{e}\otimes \id)=\delta_{e}\otimes \id=L_{\lambda(f)}(\xi).\]
Hence, by (\ref{I:injective2}), we know that $C\xi=\xi,$ i.e. $\xi\in M_{n}(\ell^{2}(G,\R)).$

%
%

\end{proof}

We close with a proposition which says that the actions we are considering are mixing. This argument is well known (see e.g. \cite{ChungLi} Lemma 5.4 or \cite{BowenLi} Theorem 4.1) and we only present it for completeness. Recall that if $G\actson X$ is an algebraic action, then the \emph{homoclinic group} of $X,$ denoted $\Delta(X),$ is defined by
\[\Delta(X)=\{x\in X:\lim_{g\to\infty}gx=1\}.\]

\begin{prop}\label{P:mixing} Let $G$ be a countable, discrete group and suppose that $f\in M_{n}(\Z(G))$ has an $\ell^{2}$ formal inverse. Then $X_{f}$ has a dense homoclinic subgroup and, in particular, the action $G\cc (X_{f},m_{X_{f}})$ is mixing.
\end{prop}

\begin{proof} Once we show that $\Delta(X_{f})$ is dense, the fact that the action is mixing is well-known, see e.g. \cite{BowenLi} Proposition 4.6. Define $q\colon \ell^{2}(G)^{\oplus n}\to (\T^{G})^{\oplus n}$ by $(q(\zeta))(l)(g)=\zeta(l)(g)+\Z$ for $1\leq l\leq n,$ $g\in G.$ Let $\xi$ be an $\ell^{2}$ formal inverse to $f.$ We claim that $q(r(\xi^{*})\Z(G)^{\oplus n})$ is a dense subgroup of $X_{f}.$ Clearly, $q(r(\xi^{*})\Z(G)^{\oplus n})$ is a subgroup of $(\T^{G})^{\oplus n}.$
By  Proposition \ref{P:basics} (\ref{I:l2adjoints}), for every  $\alpha,\beta\in \Z(G)^{\oplus n}$ we have 
\[\ip{r(f)\alpha,q(r(\xi^{*})b)}_{\T}=\ip{\widehat{r(f)\alpha},r(\xi^{*})\beta}=\ip{r(\xi)\widehat{r(f)\alpha},\widehat{\beta}}=\ip{r(\lambda(f)\xi)\widehat{\alpha},\widehat{\beta}}+\Z=\ip{\widehat{\alpha},\widehat{\beta}}+\Z=\Z.\]
This shows that $q(r(\xi^{*})\Z(G)^{\oplus n})\subseteq X_{f}.$ Since $r(\xi^{*})\C(G)^{\oplus n}\subseteq \ell^{2}(G)^{\oplus n}\subseteq c_{0}(G)^{\oplus n},$ it is clear that\\$q(r(\xi^{*})\Z(G))^{\oplus n})\subseteq\Delta(X_{f}).$~ 

	To see that $q(r(\xi^{*})\Z(G))^{\oplus n})$ is dense in $X_{f},$  suppose that $\alpha\in \Z(G)^{\oplus n}$ and $\ip{\alpha,r(\xi^{*})\beta}_{\T}=0$ for all $\beta\in \Z(G)^{\oplus n}.$ By the same computations as above, we have $\ip{r(\xi)\alpha,\widehat{\beta}}\in \Z$ for all $\beta\in \Z(G)^{\oplus n}.$ Applying this observation to $\beta=e\otimes e_{j}$ for $j=1,\dots,n$ we see that $r(\xi)\alpha\in c_{c}(G,\Z)^{\oplus n}.$ So we may write $r(\xi)\alpha=\widehat{q}$ for some $q\in \Z(G)^{\oplus n},$ and thus
\[\widehat{\alpha}=r(f)r(\xi)\widehat{\alpha}=\widehat{r(f)q}.\]
This implies that $\alpha=r(f)q\in r(f)(\Z(G)^{\oplus n}).$

\end{proof}

We remark that this proof in fact show that $G\actson (X_{f},m_{X_{f}})$ is mixing if $f$ has a ``$c_{0}$ formal inverse" (the definition of a ``$c_{0}$ formal inverse" being an obvious modification of the definition of $\ell^{2}$ formal inverse, see for example \cite[Example 4.7]{BowenLi}).

\section{Weak Equivalence}

In this section, we present the proof of Theorem \ref{T:mainintro}. By Section \ref{S:L2examples}, we know that the  actions we are concerned with are all mixing actions. Hence, by \cite{RobinAF}  these actions are free (modulo their kernels, which are finite). By Ab\'{e}rt-Weiss' theorem (see \cite{AbertWeiss}) we can focus on showing that $G\cc (X_{f},m_{X_{f}})$ is weakly contained in a Bernoulli action. Our proof is a small modification of the strategy of Bowen in \cite{BowenEntropy}:  we start by embedding $X_{f}\subseteq (\T^{G})^{\oplus n}$ and regarding $m_{X_{f}}$ as a measure on $(\T^{G})^{\oplus n}.$ Following the method in \cite{BowenEntropy}, we will write $m_{X_{f}}$ as a weak$^{*}$-limit of measures on $(\T^{G})^{\oplus n}$ which are factors of Bernoulli measures.  Unfortunately we cannot directly follow the methods in \cite{BowenEntropy}, because we are not assuming that $f$ is invertible in $M_{n}(\ell^{1}(G)),$ and so we cannot directly convolve an element in $\ell^{\infty}(G)^{\oplus n}$ with the adjoint of the inverse of $f.$ We will instead convolve with the adjoints of ``approximate inverses'' of $f$ and write the Haar measure as a weak$^{*}$-limit of such measures.

We start by proving a lemma that will do most of the heavy lifting in our prof of weak containment. If $A$ is a countable abelian group, we define $\ev_{a}\colon \widehat{A}\to \C$ by $\ev_{a}(x)=\exp(2\pi i x(a)).$ Given $\mu\in \Prob(\widehat{A}),$ recall that the Fourier transform of $\mu$ is a function $\widehat{\mu}\colon A\to \C$ defined by:
\[\widehat{\mu}(a)=\int \ev_{a}\,d\mu.\]
Given a finite set $E,$ we let $u_{E}$ be the uniform probability measure on $E.$ 

\begin{lem}\label{L:heavylifting}
Let $G$ be a countable group and fix natural numbers $m,n.$ Then there is a unique function
\[\Phi_{m}\colon M_{n}(\ell^{2}(G,\R))\to \Prob_{G}((\T^{G})^{\oplus n}),\]
such that, denoting $\mu_{m,\xi}=\Phi_{m}(\xi),$ we have
\[\widehat{\mu}_{m,\xi}(\alpha)=\prod_{\substack{g\in G,\\ 1\leq l\leq n}}\frac{\sin((2m+1)\pi (r(\xi)\alpha)(l)(g))}{(2m+1)\sin(\pi (r(\xi)\alpha)(l)(g))},\mbox{ for all $\alpha\in \Z(G)^{\oplus n}$}.\]
Moreover:
\begin{enumerate}[(i)]
\item  $\Phi_{m}$ is continuous if we give $\ell^{2}(G)^{\oplus n}$ the norm topology and $\Prob_{G}((\T^{G})^{\oplus n})$ the weak$^{*}$ topology, \label{I:weak* continuity}
\item if $\xi\in M_{n}(c_{c}(G,\R)),$ then there is a $G$-equivariant map $\rho\colon (\{-m,\dots,m\}^{G})^{n}\to (\T^{G})^{\oplus n}$ so that \label{I:Bernoulli factor}
\[\rho_{*}((u_{\{-m,\dots,m\}}^{G})^{ n})=\mu_{m,\xi}.\]
\end{enumerate}

\end{lem}

Remark: It is part of the statement of the Lemma that the product on the right hand side converges.
\begin{proof}
Define $\Psi\colon \ell^{2}(G,\R)^{\oplus n}\to \R$
by
\[\Psi(\xi)=\prod_{\substack{g\in G,\\ 1\leq l\leq n}}\frac{\sin((2m+1)\pi \xi(l)(g))}{(2m+1)\sin(\pi \xi(l)(g))}.\]
We make the following three claims:

\emph{Claim 1: for every $\xi\in \ell^{2}(G,\R)^{\oplus n},$ the product defining $\Psi(\xi)$ converges,}

\emph{Claim 2: $\Psi$ is continuous if we give $\ell^{2}(G,\R)^{\oplus n}$ the norm topology,}

\emph{Claim 3: for every $\xi\in M_{n}(c_{c}(G,\R)),$ there is a $\mu_{m,\xi}\in \Prob_{G}((\T^{G})^{\oplus n})$ with $\widehat{\mu}_{m,\xi}(\alpha)=\Psi(r(\xi)\alpha)$ for every $\alpha\in\Z(G)^{\oplus n},$ and so that $\mu_{m,\xi}$ is a factor of the Bernoulli shift measure on $(\{-m,\dots,m\}^{G})^{n}$}

Suppose we grant these three claims. Define a function $\widetilde{\Psi}\colon M_{n}(\ell^{2}(G))\to \R^{\Z(G)^{\oplus n}}$ by $\widetilde{\Psi}(\xi)(\alpha)=\Psi(r(\xi)\alpha).$ By continuity of $\Psi,$ the map $\widetilde{\Psi}$ is continuous if we give $\R^{\Z(G)^{\oplus n}}$ the product topology. By abstract Fourier analysis, the map $\mathcal{F}\colon\Prob((\T^{G})^{\oplus n})\to \C^{\Z(G)^{\oplus n}}$ sending each measure to its Fourier transform is a homeomorphism onto its image (giving $\Prob((\T^{G})^{\oplus n})$ the weak$^{*}$ topology and $\C^{\Z(G)^{\oplus n}}$ the product topology). By Claim 3, we know that $\mathcal{F}(\Prob((\T^{G})^{\oplus n}))\supseteq \widetilde{\Psi}(M_{n}(c_{c}(G,\R))).$ By continuity of $\widetilde{\Psi},$ and the fact that $\mathcal{F}(\Prob((\T^{G})^{\oplus n}))$ is closed, we find that $\widetilde{\Psi}(M_{n}(\ell^{2}(G,\R)))\subseteq \mathcal{F}(\Prob((\T^{G})^{\oplus n}).$ The fact that $\mathcal{F}$ is a homeomorphism onto its image, the inclusion $\widetilde{\Psi}(M_{n}(\ell^{2}(G,\R)))\subseteq \mathcal{F}(\Prob((\T^{G})^{\oplus n})),$ and the continuity of $\widetilde{\Psi}$ easily imply the lemma. So we focus on proving Claims 1-3.

To prove Claim 1, fix a $\xi\in \ell^{2}(G,\R)^{\oplus n}$ and set
\[f(x)=\frac{\sin((2m+1)\pi x)}{(2m+1)\sin(\pi x)}.\]
By standard real analysis, to prove Claim 1 it is enough to show that
\[\sum_{\substack{g\in G,\\ 1\leq l\leq n}}\left|1-f(\xi(l)(g))\right|<\infty.\]
Writing out the first three terms in the power series expansion of $f,$ we see that there is a constant $C>0$ so that
$|1-f(x)|\leq Cx^{2}.$
So
\[\sum_{\substack{g\in G,\\ 1\leq l\leq n}}\left|1-f(\xi(l)(g))\right|\leq C\sum_{\substack{g\in G,\\ 1\leq l\leq n}}|\xi(l)(g)|^{2}<\infty,\]
as $\xi\in \ell^{2}(G,\R)^{\oplus n}.$ This proves Claim 1.

We now prove Claim 2, we only need a slightly more sophisticated argument than that of Claim $1$ to prove Claim 2. Choose a $B>0$ so that for  any $x,y\in \R$ with $|1-x|,|1-y|<\frac{1}{2}$ we have $|\log(x)-\log(y)|\leq B|x-y|.$
Note that $f$ is an even function, so it follows by expanding $f$ in a power series that we may find $A,\delta>0$ with
$|f(x)-f(y)|\leq A|x^{2}-y^{2}|,\mbox{ if $|x|,|y|\leq \delta.$}$
We may assume that $\delta>0$ is small enough so that $|1-f(x)|<\frac{1}{2}$ for all $x\in (-\delta,\delta).$  Now suppose that we are given a sequence $\xi_{k}\in \ell^{2}(G,\R)^{\oplus n}$ and a $\xi\in \ell^{2}(G,\R)^{\oplus n}$ with $\|\xi_{k}-\xi\|_{2}\to 0.$ Without loss of generality, we may assume that $\|\xi_{k}-\xi\|_{2}<\frac{\delta}{2}$ for all $k.$ Let
$E=\left\{(l,g)\in \{1,\dots,n\}\times G:|\xi(l)(g)|\geq \frac{\delta}{2}\right\}.$
 Then $E$ is a finite set, and
\begin{align*}
\left|\sum_{(l,g)\in E^{c}}\log(f(\xi(l)(g)))-\sum_{(l,g)\in E^{c}}\log(f(\xi_{k}(l)(g)))\right|&\leq B\sum_{((l,g)\in E^{c}}|f(\xi(l)(g))-f(\xi_{k}(l)(g))|\\
&\leq AB\sum_{((l,g)\in E^{c}}|\xi(l)(g)^{2}-\xi_{k}(l)(g)^{2}|\\
&\leq AB\|\xi-\xi_{k}\|_{2}\|\xi+\xi_{k}\|_{2}\\
&\leq AB\left(\|\xi\|_{2}+\frac{\delta}{2}\right)\|\xi-\xi_{k}\|_{2}\\
&\to_{k\to\infty} 0.
\end{align*}
Exponentiating, we have:
\[\lim_{k\to\infty}\prod_{(l,g)\in E^{c}}f(\xi_{k}(l)(g))=\prod_{(l,g)\in E^{c}}f(\xi(l)(g)).\]
As $E$ is finite, 
\[\lim_{k\to\infty}\prod_{(l,g)\in E}f(\xi_{k}(l)(g))=\prod_{(l,g)\in E}f(\xi(l)(g))\]
and this proves that $\Psi$ is continuous.

It thus remains to prove Claim 3. Let $q\colon c_{c}(G,\R)^{\oplus n}\to (\T^{G})^{\oplus n}$ be defined by
$q(\xi)(l)(g)=\xi(l)(g)+\Z,$ for all $1\leq l\leq n,$ $g\in G,$
and define
$\rho\colon (\{-m,\dots,m\}^{G})^{n}\to (\T^{G})^{\oplus n}$
by $\rho(x)(l)(g)=(q(r(\xi^{*})x))(l)(g).$ Set $\mu_{m,\xi}=\rho_{*}((u_{\{-m,\dots,m\}}^{ G})^{n}).$
For any $t\in \R$ we have, by a direction computation, that
\[\int_{\{-m,\dots,m\}}e^{2\pi i tj}\,du_{\{-m,\dots,m\}}(j)=\frac{\sin((2m+1)\pi t)}{(2m+1)\sin(\pi t)},\]
(see e.g \cite{BabyRudin} section 8.13). Thus
\begin{align*}\label{E:fouriertransformcalc}
\int_{(\T^{G})^{\oplus n}}\ev_{\alpha}\,d\mu_{m,\xi}&=\int_{(\{-m,\dots,m\}^{G})^{n}}e^{2\pi i \ip{r(\xi^{*})x ,\widehat{\alpha}}}\,d(u_{\{-m,\dots,m\}}^{ G})^{ n}(x)\\
&=\int_{(\{-m,\dots,m\}^{G})^{n}}e^{2\pi i \ip{\widehat{x} ,r(\xi)\alpha}}\,d(u_{\{-m,\dots,m\}}^{ G})^{ n}(x)\\ 
&=\prod_{1\leq l\leq n}\prod_{g\in G}\frac{\sin((2m+1)\pi(r(\xi)\alpha)(l)(g)))}{(2m+1)\sin(\pi (r(\xi)\alpha) (l)(g))}. 
\end{align*}

\end{proof}

\begin{thm} Let $G$ be a countable, discrete, group and suppose that $f\in M_{n}(\Z(G))$ has an $\ell^{2}$ formal inverse. Then $G\cc (X_{f},m_{X_{f}})$ is weakly equivalent to a Bernoulli shift.
\end{thm}

\begin{proof}
By Theorem 1.4 of \cite{RobinAF}, and  Proposition \ref{P:mixing}, we know that $G\actson (X_{f},m_{X_{f}})$ is free modulo its kernel, and that this kernel is finite. It thus follows from  \cite{AbertWeiss} that $G\cc (X_{f},m_{X_{f}})$ weakly contains any Bernoulli shift. So we only present the proof that $G\actons (X_{f},m_{X_{f}})$ is weakly contained in any Bernoulli shift.  We adopt notation as in Lemma \ref{L:heavylifting}.  Let $\xi$ be an $\ell^{2}$ formal inverse to $f.$ By Proposition \ref{P:basicpropertiesofformalinverse} (\ref{I:tha realness}), we may find a seqeuence $\xi_{k}\in M_{n}(c_{c}(G,\R))$ so that $\|\xi_{k}-\xi\|_{2}\to 0.$ By part (\ref{I:Bernoulli factor}) of Lemma \ref{L:heavylifting}, we know that each $G\actson ((\T^{G})^{\oplus n}, \mu_{m,\xi_{k}})$ is weakly contained in any Bernoulli shift. So by Proposition 5.2 of \cite{BowenEntropy} and Lemma \ref{L:heavylifting} (\ref{I:weak* continuity}), it is enough to show that
\begin{equation}\label{E:its what we have to prove yo!}
\lim_{m\to\infty}\mu_{m,\xi}=m_{X_{f}} \mbox{ weak$^{*}$}.
\end{equation}
 Using  Lemma \ref{L:heavylifting} and the fact that $\{\ev_{\alpha}:\alpha\in \Z(G)^{\oplus n}\}$ spans a norm dense subset of $(\T^{G})^{\oplus n},$ to prove (\ref{E:its what we have to prove yo!}) we simply have to show that
\[\lim_{m\to\infty}\prod_{1\leq l\leq n}\prod_{g\in G}\frac{\sin((2m+1)\pi(r(\xi)\alpha)(l)(g)))}{(2m+1)\sin(\pi(r(\xi)\alpha)(l)(g))}=\begin{cases}
1,&\textnormal{ if $\alpha\in r(f)(\Z(G)^{\oplus n}),$}\\
0,&\textnormal{ if $\alpha\notin r(f)(\Z(G)^{\oplus n}).$}
\end{cases}\]

	Suppose that $\alpha\in r(f)(\Z(G)^{\oplus n})$ and write $\alpha=r(f)\beta$ with $\beta\in \Z(G)^{\oplus n}.$ Then $r(\xi)\alpha=r(\xi)r(f)\beta=r(\lambda(f)\xi)\beta=\widehat{\beta}$,
so $r(\xi)\alpha\in c_{c}(G,\Z)^{\oplus n}$, and 
	\[\prod_{1\leq l\leq n}\prod_{g\in G}\left(\frac{\sin((2m+1)\pi(r(\xi)\alpha)(l)(g))}{(2m+1)\sin(\pi(r(\xi)\alpha)(l)(g))}\right)=1\]
for every $m.$

	Suppose that $\alpha\notin r(f)(\Z(G)^{\oplus n}).$ By Proposition \ref{P:basicpropertiesofformalinverse} (\ref{I:tha realness}) we have $r(f)r(\xi)\alpha=\widehat{\alpha},$ and since  $\alpha\notin r(f)(\Z(G)^{\oplus n}),$ we must have that $r(\xi)\alpha\notin c_{c}(G,\Z)^{\oplus n}$ . Hence we can find an integer $1\leq l_{0}\leq n$, and a $g_{0}\in G$ so that $(r(\xi)\alpha)(l_{0})(g_{0})\notin \Z.$ As
	\[\int e^{2\pi i j(r(\xi)\alpha)(l)(g) }\,du_{\{-m,\dots,m\}}(j)=\frac{\sin((2m+1)\pi(r(\xi)\alpha)(l)(g))}{(2m+1)\sin(\pi(r(\xi)\alpha)(l)(g))}\]
we have that
	\[\left|\frac{\sin((2m+1)\pi(r(\xi)\alpha)(l)(g))}{(2m+1)\sin(\pi(r(\xi)\alpha)(l)(g)))}\right|\leq 1,\]
for every $1\leq l\leq n,g\in G,m\in \N.$
	So
	\[\left|\prod_{1\leq l\leq n}\prod_{g\in G}\left(\frac{\sin((2m+1)\pi(r(\xi)\alpha)(l)(g))}{(2m+1)\sin(\pi(r(\xi)\alpha)(g))}\right)\right|\leq \left|\frac{\sin((2m+1)\pi(r(\xi)\alpha)(l_{0})(g_{0}))}{(2m+1)\sin(\pi(r(\xi)\alpha)(l_{0})(g_{0}))}\right|\to_{m\to\infty}0,\]
	the last step following as $(r(\xi)\alpha)(l_{0})(g_{0})\notin \Z.$ This completes the proof.

\end{proof}

%

\end{document}